\documentclass[12pt,a4paper]{amsart}

\usepackage{textcomp}

\usepackage[margin=1.0in]{geometry}
\usepackage{setspace}
\usepackage{xypic}
\usepackage{amsfonts, amsthm, amssymb, amsmath, stmaryrd}
\usepackage{mathrsfs,array}
\usepackage{eucal,times,color,enumerate}
\usepackage{url}
\usepackage{float}
\usepackage[all]{xy}
\usepackage{graphicx}
\usepackage{ifpdf}
\ifpdf
\fi

\usepackage{color}
\usepackage{mathrsfs}
\usepackage{amssymb}
\usepackage{hyperref}

\hypersetup{
  colorlinks   = true,
  urlcolor     = red,
  linkcolor    = blue,
  citecolor   = red
}

\newtheorem{theorem}{Theorem}[section]
\newtheorem*{theorem*}{Theorem}
\newtheorem{proposition}[theorem]{Proposition}
\newtheorem{lemma}[theorem]{Lemma}
\newtheorem{corollary}[theorem]{Corollary}

\newtheorem{example}[theorem]{Example}
\newtheorem{remark}[theorem]{Remark}

\newtheorem*{lemma*}{Lemma}
\newtheorem*{remark*}{Remark}
\newtheorem*{example*}{Example}

\newcommand{\mb}[1]{\mathbb{#1}}
\newcommand{\Z}{\mathbb{Z}}
\newcommand{\supp}{\operatorname{Supp}}
\newcommand{\ex}{\operatorname{Ex}}
\newcommand{\C}{\mathbb C}
\newcommand{\f}{\tilde \phi}

\newcommand{\de}{\operatorname{det}}
\newcommand{\pn}{\operatorname{\rho}}
\newcommand{\degree}{\operatorname{deg}}
\newcommand{\NS}{\operatorname{NS}}
\newcommand{\rk}{\operatorname{rk}}

\makeatletter\@addtoreset{equation}{section}\makeatother

\makeatletter\@addtoreset{subsection}{equation}\makeatother

\begin{document}

\pagestyle{plain}

\title{Rational maps and $\mathrm{K3}$ surfaces}

\author{Ilya Karzhemanov and Grisha Konovalov}

\begin{abstract}
For a very general complex projective $\mathrm{K3}$ surface $S$
and a smooth projective surface $A$ with trivial canonical class,
we prove that there is no dominant rational map $A \dasharrow S$,
which is not an isomorphism.
\end{abstract}

\address{\newline{\normalsize Laboratory of AGHA, Moscow Institute of Physics and Technology, 9 Institutskiy per., Dolgoprudny,
Moscow Region, 141701, Russia}
\newline{\it E-mail address}: {\tt karzhemanov.iv@mipt.ru}}

\medskip

\address{\newline{\normalsize  HSE University, 6 Usacheva str., Moscow, 119048, Russia}
\newline{\it E-mail address}: {\tt gkonovalov@hse.ru}}

\thanks{{\it MS 2010 classification}: 14J28, 14E05, 14J10}


\sloppy

\maketitle

\bigskip

\section{Introduction}

The purpose of this note is to contribute to the study of rational
maps with target a $\mathrm{K3}$ surface (we will mainly consider
\emph{smooth complex projective} surfaces). A particular question
is this:

$$
\begin{array}{c}
\textit{Does the existence of a dominant rational map, which is
\emph{not an isomorphism}, from a surface with}
\\
\textit{trivial canonical class to a given $\mathrm{K3}$ surface
constrain this $\mathrm{K3}$ surface from being ``general
enough''?}
\end{array}
$$

The answer is provided by the following

\begin{theorem}
\label{main} Suppose there exists a dominant rational map $\phi
\colon A \dasharrow S$ from a surface $A$ with trivial canonical
class to a $\mathrm{K3}$ surface $S$. Suppose also that $\degree
\phi \ge 2$. Then the Picard number $\pn(S)$ is \emph{strictly
greater} than $1$.
\end{theorem}

The study of dominant rational maps between $\mathrm{K3}$ and
Abelian surfaces goes back to the papers \cite{S-I} and \cite{Mo}
(there the authors used the lattice\,-\,theoretic point of view).
Also, a special case of Theorem~\ref{main}, concerning general
quartic surfaces, follows from the results obtained by C.~Voisin
in \cite[Theorem 2]{Vo} (she took the IVHS approach). We unify
these trends in a way by employing some simple birational and
projective geometry (cf. Remark~\ref{lattices} below).

Let us collect some examples supporting Theorem~\ref{main}.

\begin{example}[Kummer surfaces]
\label{kummer} Take an Abelian surface $A$ and consider the
involution $\tau \colon A \longrightarrow A$ acting as $x \mapsto
-x$ for all $x \in A$. The minimal resolution $\mathrm{Km}\,A
\longrightarrow A/\tau$ of the quotient $A/\tau$ (extracting
sixteen $(-2)$\,-\,curves out of ordinary double points) gives a
$\mathrm{K3}$ surface admitting a rational map $A \dasharrow
\mathrm{Km}\,A$ of degree $2$. One has $\rho(\mathrm{Km}\,A) = 16
+ \rho(A)$ in this case.
\end{example}

In view of Example~\ref{kummer} it is worth noting that already
the case of $\mathrm{K3}$ surfaces in \emph{positive
characteristic} ($\ge 5$) indicates the setup of
Theorem~\ref{main} is not so trivial. Indeed, it is known that the
supersingular Kummer surfaces form a $1$\,-\,dimensional family,
whereas there is a $9$\,-\,dimensional family of all supersingular
$\mathrm{K3}$ (see \cite[Section 5]{L} for precise results and
references). This, together with the ``rational sandwich theorem''
from \cite{L1}, implies that there exist \emph{non\,-\,Kummer}
(supersingular) $\mathrm{K3}$ surfaces dominated by Abelian ones.

Here are more examples over $\C$:

\begin{example}[Shioda surfaces]
\label{example_sh} Let $S$ be an elliptic $\mathrm{K3}$ surface
with a section and two singular fibers of type
$\uppercase\expandafter{\romannumeral 2}{^*}$ (plus some other
singular fibers). It was proved by T.~Shioda in \cite{Shi} that
there exists a Kummer surface $\mathrm{Km}\,C_1 \times C_2$, where
$C_i$ are elliptic curves, and two rational maps $\mathrm{Km}\,C_1
\times C_2 \dasharrow S$, $S \dasharrow \mathrm{Km}\,C_1 \times
C_2$ of degree $2$. One can also compute that $\rho(S) \ge 18$.

The existence of such $S$ follows from the results in \cite{N}.
Namely, all $\mathrm{K3}$ surfaces $Y$, admitting a
\emph{primitive} embedding of lattices
$$
U \oplus E_8(-1)^{\oplus 2} \oplus \Z\left<-2d\right>
\hookrightarrow \NS(Y)
$$
for some $d \ge 1$, form a moduli space of dimension one. Any such
$Y$ possesses a jacobian elliptic fibration (that is it has a
section) and two singular fibers of type
$\uppercase\expandafter{\romannumeral 2}{^*}$. Also for very
general such $Y$ the equality $U \oplus E_8(-1)^{\oplus 2} \oplus
\Z\left<-2d\right> = \NS(Y)$ holds. Moreover, the lattice $U
\oplus E_8(-1)^{\oplus 2} \oplus \Z\left<-2d\right>$ has a
\emph{unique}, up to isometries, primitive embedding into the
$\mathrm{K3}$ lattice (see \cite[Theorem 1.14.4]{N1}).
Consequently, the transcendental lattice $T(Y)$ coincides with $U
\oplus \Z\left<2d\right>$, and it follows from \cite[Corollary
4.4]{Mo} that $Y$ is not a Kummer surface.

Finally, let us indicate that for \emph{even} $d$ there is a
geometric construction of $Y$ due to K.~Hulek and M.~Sch\"utt, see
\cite{H-S}. They proved that any such $Y$ is obtained by a
quadratic base change from a rational (jacobian) elliptic surface.
It was also shown that $Y$ represents a specific member of the
so\,-\,called \emph{Barth -- Peters family} (cf. \cite[Lemma
4.8]{H-S}).
\end{example}

\begin{example}[Symplectic automophisms]
An automorphism $\sigma$ of a $\mathrm{K3}$ surface $A$ is called
\emph{symplectic} if $\sigma$ induces trivial action on $H^0(A,
\Omega_A^2) \simeq \C$. Let $G \subset \mathrm{Aut}(A)$ be a
finite group of symplectic automorphisms. The minimal resolution
$S \longrightarrow A/G$ provides a $\mathrm{K3}$ surface $S$
together with a dominant rational map $A \dasharrow S$. It is easy
to see that $\rho(S) \ge 9$. Indeed, the morphism $A
\longrightarrow A/G$ can not be \'etale, hence the resolution $S
\longrightarrow A/G$ has an exceptional locus. Note that with a
detailed analysis of possible $G$ one can describe the Picard
group of $S$ more precisely (see \cite[Chapter 15]{Hu}).
\end{example}

\begin{example}
\label{minus-two} Let $S \subset \mb P^1 \times \mb P^2$ be given
by a general divisor of type $(2,3)$. Projection on the first
factor yields an elliptic fibration $\pi: S \longrightarrow \mb
P^1$. Then there exists a rational map $S \dasharrow S$ of degree
$16$ which induces the morphism $[4]: E \longrightarrow E$ on
every smooth fiber $E$ of $\pi$ (cf. \cite{D}). Note that $\rho(S)
= 2$ and so the estimate in Theorem~\ref{main} is sharp. Also
observe that the quadratic form on $\NS(S)$ is $2x^2 + 6xy$ and
hence $S$ does not contain $(-2)$\,-\,curves --- contrary to the
claim in the first version of our paper. We are indebted to S.
Galkin and E. Shinder for communicating this construction to us.
\end{example}

\begin{remark}
\label{lattices} Suppose that $A$ in Theorem~\ref{main} is an
Abelian surface and take the resolution of indeterminacies of
$\phi$ as in the diagram \eqref{diag} below. We have $T(W) = T(A)$
for the transcendental lattices of $A$ and its blow\,-\,up $W$. It
also follows from the projection formula for the resolved morphism
$q: W \longrightarrow S$ that $q^*$ gives an embedding $T(S)
\hookrightarrow T(W)$. Thus $\rk\,T(S) \le \rk\,T(A) \le 5$ and so
$\rho(S) \ge 17$ (we would like to thank E. Shinder for pointing
out this argument). Yet let us note that our result provides more
impact on the geometry of $S$. In particular, when $A = S$ is
\emph{general}, Theorem~\ref{main} improves the main result of
\cite{Ch} (our proof is also considerably shorter). We refer to
\cite{B-S-V}, \cite{Ma}, \cite{N2} and \cite{A-R-V}, \cite{B},
\cite{Ka} for relevant studies of dominant rational maps between
different $\mathrm{K3}$ surfaces and some other classes of
algebraic varieties.
\end{remark}

\section{The proof of Theorem~\ref{main}}

Consider the following commutative diagram:

\begin{equation}
\label{diag} \xymatrix{
& W \ar[dl]_p \ar[dr]^q  \ar[rr]^g && B \ar[dl]^f \\
A \ar@{-->}[rr]^{\phi} && S }
\end{equation}

Here $p$ blows up the indeterminacy locus of $\phi$, $g$ and $f$
are a proper contraction and a finite morphism, respectively,
provided by the Stein factorization of $q$, and $B$ is a normal
surface.

The idea behind the proof of Theorem~\ref{main} is that the
exceptional set $\ex(p)$ of $p$ should constrain the geometry of
$S$ and $\phi$.

The next lemma shows that $\phi$ does not contract any curves:

\begin{lemma}
\label{exceptional} The inclusion $\ex(g) \subseteq \ex(p)$ holds.
\end{lemma}

\begin{proof}
Restrict $\phi$ to appropriate $A^\prime := A \setminus \{\text{a
finite set of points}\}$ to obtain a \emph{morphism} $\f \colon
A^\prime \longrightarrow S$. Consider the exact sequence
$$
0 \to \f^*\Omega^1_S \xrightarrow{\gamma} \Omega^1_{A^\prime} \to \Omega^1_{\f} \to 0
$$
of sheaves of differentials. It follows that
$$
\supp \Omega^1_{\f} = V(\de \gamma), \text{the scheme of zeros},
$$
for $\de\,\gamma$ being the morphism $\f^*K_S \longrightarrow
K_{A^\prime}$ induced by $\gamma$. Further, since the canonical
classes of $A$, $A^\prime$ and $S$ are trivial, we get that $\de
\gamma$ is a constant. Moreover, we have $\de \gamma \ne 0$, since
$\gamma$ induces an isomorphism of stalks over the generic point.
Thus we obtain $\Omega^1_{\f} = 0$. Then, since $q\big\vert_{W
\setminus \ex(p)} = \f$ for a natural identification $A^\prime = W
\setminus \ex(p)$, we obtain
\begin{equation}
\label{omega-q} \supp \Omega^1_q \subseteq \ex(p)
\end{equation}
for the relative cotangent sheaf $\Omega^1_q$ of $q$.

Note that $\ex(g) = \supp \Omega^1_g$ because $g$ is a proper
contraction. Note also that $\supp \Omega^1_g \subseteq \supp
\Omega^1_q$ by the exact sequence
$$
g^*\Omega^1_f \to \Omega^1_q \to \Omega^1_g \to 0
$$
(cf. \eqref{diag}). The inclusion $\ex(g) \subseteq \ex(p)$ of
exceptional loci now follows from \eqref{omega-q}.
\end{proof}

Write $\ex(p) = \displaystyle\bigcup_i E_i$ for smooth rational
exceptional curves $E_i$. Let also $R := \displaystyle\sum_j Z_j$
be the ramification divisor of $f$ with $Z_j \subset B$ some (not
necessarily distinct) irreducible curves. This $R$ is defined in
terms of the canonical classes as follows:
\begin{equation}
\label{ram-div} R\big\vert_{B \setminus g(\ex(g))} = K_{B
\setminus g(\ex(g))} \otimes (f^{*}K_S^{\vee})\big\vert_{B
\setminus g(\ex(g))} = K_{B \setminus g(\ex(g))}.
\end{equation}
We can also put $B \setminus g(\ex(g)) = W \setminus \ex(g)$ and
identify $R$ with its closure in $W$.

\begin{lemma}
\label{support} $R \subset W$ consists of exactly those $E_i$ that
are \emph{not} contracted by $g$.
\end{lemma}

\begin{proof}
We have $K_W = \displaystyle\sum_i n_iE_i$ for some $n_i \in
\mathbb{N}$ by the ramification formula applied to $p$ (recall
that $K_A$ is trivial). Now the claim follows from
$$
R\big\vert_{B \setminus g(\ex(g))} = K_{B \setminus g(\ex(g))} =
\sum_i n_iE_i\big\vert_{W \setminus \ex(g)},
$$
where the first identity is due to the Hurwitz formula applied to
the finite morphism $f: B \longrightarrow S$, together with the
fact that $K_S$ is trivial.
\end{proof}

The following lemma proves a special case of Theorem \ref{main}:

\begin{lemma}
\label{ramification} Suppose that $\supp f^*(f(Z_j)) \subseteq R$
for some $j$. Then $f(Z_j)$ is a $(-2)$\,-\,curve on $S$ and hence
$\rho(S) \ge 2$.
\end{lemma}

\begin{proof}
We have $\supp q^*f(Z_j) \subseteq \supp g^*R$ by assumption. Note
further that $\supp g^*R \subseteq \ex(p)$ by
Lemmas~\ref{exceptional} and \ref{support}. Consequently, we
obtain $\supp q^*f(Z_j) \subseteq \ex(p)$ and
$$
(\degree q)(f(Z_j)^2) = (q^*f(Z_j)^2) < 0,
$$
where the latter inequality is due to the fact that the matrix of
$(E_i \cdot E_k)$ is negative definite (see \cite{M}). The claim
now follows because the arithmetic genus $p_a(f(Z_j)) =
\displaystyle\frac{1}{2}(f(Z_j)^2) + 1 \ge 0$ and the equality
holds iff $f(Z_j) \simeq \mb P^1$.
\end{proof}

Let us turn to the case not covered by Lemma~\ref{ramification}:

\begin{lemma}
\label{quasifinite} Suppose that either $R = \emptyset$ or for
every $j$ there exists an effective cycle $\widetilde{Z_j} \ne 0$
such that $f^*(f(Z_j)) \ge Z_j + \widetilde{Z_j}$ and $\supp
\widetilde{Z_j} \nsubseteq R$. Then $\phi$ is \emph{onto}
$S^\prime := S \setminus \{\text{a finite set of points}\}$ and is
unramified over $S^\prime$.\footnote{~We use interchangeably the
terms ``unramified'' and ``smooth'' when applied to a finite
morphism. Note also that in our case $\phi: A^\prime
\longrightarrow S^\prime$ \emph{need not} be a topological
covering (the picture here is similar to e.\,g. $f: \C
\longrightarrow \C$, given by $f(z) = z^2(z - 1)$, which becomes
smooth on $\C\setminus{\{0\}}$).}
\end{lemma}

\begin{proof}
If $R = \emptyset$, then Lemmas~\ref{exceptional}, \ref{support}
imply that $\phi$ induces a \emph{finite surjective unramified
morphism} $A^\prime \longrightarrow S^\prime$, where $\prime$
indicates removing a finite set of points.

In the second case, when $R \ne \emptyset$, Lemma~\ref{support}
and the assumption imply that each $f(Z_j)$ is dominated by
$p(g_*^{-1}(\widetilde{Z_j}))$ via $\phi$ because
$g_*^{-1}(\widetilde{Z_j}) \nsubseteq \ex(p)$. Then it follows
from Lemma~\ref{exceptional} and the equality $q(\ex(p)) =
f(\displaystyle\bigcup_j Z_j)$ that $\phi$ induces a \emph{finite
surjective morphism} $A^\prime \longrightarrow S^\prime$ for
$A^\prime = W\setminus\ex(p)$ (cf. \eqref{diag}). Finally,
$\phi\big\vert_{A^\prime}$ is unramified, since $R \subseteq
\ex(p)$ (cf. \eqref{ram-div}).
\end{proof}

We will now assume that $\rho(S) = 1$ and complete the proof of
Theorem \ref{main} by establishing a contradiction (this should be
contrasted with Example~\ref{minus-two}). Write $\NS(S) = \Z \cdot
H_S$ for an ample generator $H_S$.

It follows from Lemmas~\ref{ramification} and \ref{quasifinite}
that $\phi$ induces a surjective unramified morphism $A^\prime
\longrightarrow S^\prime$ as above. Then, since the complements $A
\setminus A^\prime$ and $S \setminus S^\prime$ have codimension
$\ge 2$, the divisorial pull\,-\,back $\phi^*H_S$ of $H_S$ to $A$
is naturally defined. Namely, write $H_S$ as a sum $\displaystyle
\sum_i d_iD_i$ of prime Weil divisors $D_i$ with $d_i \in
\mathbb{Z}$, set $\phi^*D_i$ to be the Zariski closure of
$\phi^{-1}(D_i \cap S^\prime)$ and extend by linearity. Note that
$\phi^*$ preserves the linear equivalence.

So we have $\phi^*H_S = H_A$ for some divisor $H_A$ on $A$. Let $n
\in \mathbb{N}$ be such that $nH_S$ is \emph{very ample}. From the
construction we get
$$
n^2(H_A^2) = n^2(\deg \phi)(H_S^2) > 0
$$
and $(nH_A \cdot Z) = (nH_S \cdot \phi_*(Z)) > 0$ for every curve
$Z \subset A$. Hence $H_A$ is ample by the Nakai -- Moishezon
criterion. Then multiplying by a common factor one may assume that
\emph{both} $H_A$ and $H_S$ are very ample.

Consider a general smooth curve $C \in |H_A|$ with $C \subset
A^\prime$. Then the (scheme\,-\,theoretic) image $\phi(C) \subset
S$ is defined. Furthermore, since actually $\phi(C) \subset
S^\prime$, the preimage $\phi^{-1}\phi(C)$ is defined as well.

\begin{proposition}
\label{smooth} $\phi(C)$ is a smooth curve.
\end{proposition}

\begin{proof}
Suppose that $\phi(C)$ is singular at some point $p$. Then, since
$\phi\big\vert_{A^\prime}$ is a local isomorphism, there exist two
\emph{distinct} points $p_1, p_2 \in C$ and tangent vectors $t_i
\in T_{p_i}C$ such that $\phi(p_i) = p$ and $\phi_*(t_1) \ne
\phi_*(t_2)$ (cf. Figure~\ref{fig-1}).

\begin{figure}
\includegraphics{picture.1}
\caption{~}\label{fig-1}
\end{figure}

Consider a general curve $H \in \phi^*|H_S| \subset |H_A|$ which
is tangent to $C$ along both $t_i$ (these $H$ constitute a
\emph{codimension $3$} linear system $\mathcal{L}_A \subset
\phi^*|H_S|$ by generality of $C$). Then the curve $\phi(H \cap
A^\prime)$ is singular at $p$.\footnote{~Note that a priori $H
\not\subset A^\prime$ and $\phi$ has indeterminacies on $H$.} The
linear system $\mathcal{L}_S \subset |H_S|$ of the closures of
such $\phi(H \cap A^\prime)$ has codimension $3$, since $\phi^*H_S
= H_A$, and consists of \emph{all} members from $|H_S|$ singular
at $p$. Indeed, these members are cut out on $S$ by hyperplanes
containing the tangent plane $T_pS$, which is a codimension $3$
condition.

On the other hand, take a general curve $\widetilde{H} \in
\phi^*|H_S|$ through $p_1,p_2$, which has tangency along $t_1$,
but not along $t_2$ (i.\,e. $\widetilde{H}$ is the preimage of a
curve from $|H_S|$ tangent to only one branch of $\phi(C)$ at
$p$). Then $\phi(\widetilde{H} \cap A^\prime)$ is singular at $p$
for $\phi_*(t_1) \ne \phi_*(t_2)$ and $\phi\big\vert_{A^\prime}$
being a local isomorphism, hence $\widetilde{H} \in
\mathcal{L}_A$, a contradiction.
\end{proof}

\begin{remark}
\label{remark:dedieu} Using \cite[Proposition 2.1, (ii)]{D}
together with our Proposition~\ref{smooth} it is tempting to claim
a contradiction and conclude the proof of Theorem~\ref{main}. But
this, however, disagrees with Example~\ref{minus-two} because the
argument in \cite{D} applies to \emph{arbitrary} $A$, $S$ and
$\phi$ of degree $\ge 2$. The reason of this confusion is that
\cite[Proposition 2.1, (ii)]{D} is formulated a bit imprecisely. A
more accurate formulation should be that $\phi(C)$ is \emph{at
most} nodal. In fact, Dedieu considers the scheme $J$
parameterizing various pairs $(C^*, x_1 + x_2)$, where $C^* \in
|H_A|$, $x_i \in C^*$ are distinct points with $\phi(x_1) =
\phi(x_2)$, and shows that $\dim J = \dim |H_A|$. Then the fiber
of the natural projection $J \longrightarrow |H_A|$ over a general
$C$ is either \emph{finite} (in which case $\phi(C)$ is nodal and
$\phi\big\vert_C$ is its normalization) or \emph{empty} (in which
case $\phi\big\vert_C$ is an isomorphism).
\end{remark}

\begin{corollary}
\label{conn} $\phi^{-1}\phi(C) = C$.
\end{corollary}

\begin{proof}
Firstly, the closure of the preimage $\phi^{-1}\phi(C)$ is
connected, since $\phi^*H_S = H_A$ is ample. Now, if
$\phi^{-1}\phi(C)$ has an irreducible component $C' \ne C$, then
there is a point $p \in C \cap C'$, so that $\phi(C)$ is singular
at $\phi(p)$ for $\phi\big\vert_{A^\prime}$ being a local
isomorphism (cf. Figure~\ref{fig-1}). But the latter contradicts
Proposition~\ref{smooth}.
\end{proof}

We have $\phi(C) \in |kH_S|$ for some $k \ge 1$ because $\rho(S) =
1$ by assumption. Now observe that $\phi_*C = \phi_*\phi^*H_S =
(\degree \phi)H_S$ by the projection formula and also $\phi_*C =
d\phi(C) = dkH_S$ for $d := \degree \phi\big\vert_C$. Hence
$\degree \phi = dk$. On the other hand, we have $d = \degree \phi$
by Corollary~\ref{conn}, and so $k = 1$. Thus $\phi(C) \in |H_S|$
and it follows that $|H_A| = \phi^*|H_S|$ for $C =
\phi^{-1}\phi(C)$. In particular, $\phi$ must be \emph{regular}
because $C \subset A^\prime$ by construction, hence it induces an
\emph{\'etale} cover, which is impossible for the $\text{K3}$
surface $S$.

Theorem \ref{main} is completely proved.

We conclude the paper by asking the following question (suggested
by the referee):

$$
\begin{array}{c}
\textit{Is there an analog of Theorem \ref{main} for
\emph{non\,-\,projecitve} compact complex surfaces?}
\end{array}
$$

\bigskip

\thanks{{\bf Acknowledgments.}
We would like to thank S. Galkin for introducing us to the
subject, treated in Theorem~\ref{main}, and F. Bogomolov, A.
Bondal, I. Dolgachev, V. Nikulin, Y. Zarhin, I. Zhdanovskiy for
their interest and valuable comments. We are also grateful to an
anonymous referee whose remarks and suggestions have improved the
exposition of our paper. This work was supported by the Priority
2030 Strategic Academic Leadership Program and by the HSE
University Basic Research Program.

\newpage

\begin{center}
\text{\huge\bf Corrigendum}\\
\medskip
{\small to the paper \cite{KK}}\\
\end{center}

\bigskip

\bigskip

\bigskip

(i) \emph{\bf On pp. 306 -- 307}: we considered the identity
$\phi^*H_S = H_A$ with very ample $H_S$ and argued that the
divisor $H_A$ is ample, using the projection formula. This
argument is not valid in general for irregular $\phi$. The issue
can be corrected as follows: from our constructions we get
$$
H_A = \phi^*(H_S) = p_*q^*(H_S)
$$
(cf. diagram (2.1)); we also have
$$
q^*(H_S)^2 = (\deg q)(H_S^2) > 0
$$
and $(q^*(H_S) \cdot Z) = (H_S \cdot q_*(Z)) > 0$ for any curve $Z
\subset W$ not contained in $\text{Ex}(q)$. It follows that the
divisor $q^*(H_S)$ is nef and big. Moreover, since the linear
system $q^*(|H_S|)$ is basepoint-free, Lemma 2.1 and Nakai --
Moishezon criterion imply that $H_A = p_*q^*(H_S)$ is ample.
Indeed, generic divisor $D$ from the \emph{mobile} linear system
$p_*q^*(|H_S|)$ satisfies
$$
(D^2) > 0, (D \cdot Z) > 0
$$
for any curve $Z \subset A$ (recall that $p$ is a sequence of
blow-ups and $p_*$ is the birational transform).

\bigskip

(ii) \emph{\bf About the linear system $\mathcal{L}_A$ from the
proof of Proposition 2.5}: this $\mathcal{L}_A$ is contained in
$\phi^*|H_S|$ and the term ``codimension'' applies to
$\mathcal{L}_A \subset \phi^*|H_S|$. Hence every $H \in
\mathcal{L}_A$ is (generically) a preimage of some curve from $S$.
But then, once $H$ passes through some point in $\phi^{-1}(p)$, it
must pass through \emph{all} points in this fiber. In particular,
passing through $p_1,p_2 \in \phi^{-1}(p)$ imposes \emph{only one}
linear condition on $\mathcal{L}_A$, since this holds for the
whole $\phi^*|H_S|$. Two other linear conditions are obtained from
$H$ being tangent along $t_1$ and $t_2$, which is the same as to
pass through two, infinitely close, points over $p_1$ and $p_2$,
respectively. All in all one gets codimension $3$ as in the text.

Descending $\mathcal{L}_A$ to $S$, yields the codimension $3$
linear subsystem $\mathcal{L}_S \subset |H_S|$, which consists of
sections of $S$ by hyperplanes containing the tangent plane $T_pS
= \left<\phi_*(t_1),\phi_*(t_2)\right>$. This implies that
$\mathcal{L}_S$ is comprised of \emph{all} curves in $|H_S|$
singular at $p$. Indeed, any curve in $|H_S|$ is singular at $p$
iff the corresponding hyperplane passes through $T_pS$, and this
is also a codimension $3$ condition.

\bigskip

(iii) \emph{\bf About Remark 2.6 and the paper \cite{D}}: the
arguments in the proof of \cite[Proposition 2.1,\,(ii)]{D} imply
the following

\newpage

{\bf Claim.}\footnote{This was pointed out to us by an anonymous
referee.} {\it Let $\phi: A \dashrightarrow S$ be a dominant
rational map of $\deg \phi \ge 2$ between two smooth projective
surfaces. Then for a sufficiently ample line bundle $H_A$ on $A$
and a general member $C \in |H_A|$ the image $\phi(C)$ is a
\emph{singular} curve. Here ``$H_A$ being sufficiently ample''
just requires any $0$-dimensional subscheme $\xi =
\{p_1,p_2,q_1,q_2\}$ in $A$ of length $4$, with $\phi(p_1) =
\phi(p_2),\phi(q_1) = \phi(q_2)$ and generic point $(p_1,p_2)
\times (q_1,q_2) \in (A \times_S A)^2$, to impose independent
linear conditions on $|H_A|$. The latter means that the
restriction map
\begin{equation}
\label{surj} H^0(A,\mathcal{O}_A(H_A)) \to
H^0(\xi,\mathcal{O}_{\xi}(H_A))
\end{equation}
is surjective.}

\medskip

The arguments from \cite{D} appeal to general points on $A$ and
are based on a simple parameter count. Hence they equally apply
and also give {\bf Claim} in the following setting:

\begin{itemize}

    \item the surfaces $S$ and $A$ are \emph{quasi-projective};

    \smallskip

    \item in place of (complete) $|H_A|$ one takes some linear system
    $\mathcal{L}$ satisfying \eqref{surj} (with appropriate notational adjustments).

\end{itemize}

Here is an elementary example showing that \emph{generic $\phi(C)$
can be smooth}:\footnote{A more elaborate example (not to mention
Proposition 2.5 from our paper) can be constructed for an Abelian
surface $A = S$ with isogeny $\phi$.}

\medskip

Take $A = S = \mathbb{A}^2_{x,y}$, $\mathcal{L}$ the linear system
of curves
$$
C:\ ay = bx^3 + cx^2 + dxy + e, \ a,b,c,d,e \in \mathbb{C},
$$
and
$$
\phi: (x,y) \mapsto (x,y^2).
$$
Then the curve $\phi(C) \subset S$ is given by equation
\begin{equation}
\label{su-qua} y(a - dx)^2 = b^2x^6 + 2bcx^5 + c^2x^4 + 2bex^3 +
2cex^2 + e^2.
\end{equation}

\medskip

\underline{\it Smoothness of generic $\phi(C)$.} Taking $y$- and
$x$-partial derivatives of the polynomial in \eqref{su-qua} and
equating them to zero, gives
$$
a - dx = 0, \qquad 6b^2x^5 + 10bcx^4 + 4c^2x^3 + 6bex^2 + 4cex =
0.
$$
Altogether the latter equations impose a \emph{non-trivial}
algebraic relation between the variables $a,b,c,d,e$. Hence
generic $\phi(C)$ is smooth.

\medskip

\underline{\it Condition \eqref{surj} for $\mathcal{L}$.} Take a
length $4$ scheme $\xi = \{(x_1,\pm y_1), (x_2,\pm y_2)\}$ as
above. Then the inclusion $\xi \subset C$ yields the following
linear relations in $\mathcal{L}$:
\begin{eqnarray}
\nonumber ay_1 - dx_1y_1 = 0,\\
\nonumber
bx_1^3 + cx_1^2 + e = 0,\\
\nonumber
ay_2 - dx_2y_2 = 0,\\
\nonumber bx_2^3 + cx_2^2 + e = 0.
\end{eqnarray}
These are linearly independent, since $y_i \ne 0, x_1 \ne x_2$ and
$x_i \ne 0,1$ due to the generality condition for $\xi$. Hence
$\mathcal{L}$ satisfies \eqref{surj}.
\\

Thus we see that {\bf Claim} is \emph{false}. This clarifies the
``at most nodal'' phrase in our Remark 2.6.

\medskip

{\it Remark.} Let us briefly explain where the reasoning of
\cite{D} goes wrong. Firstly, the fiber product $S \times_S S$
considered in {\it loc.\,cit.} belongs to $S_{\bf x} \times S_{\bf
y}$ rather than to the Hilbert scheme $S^{[2]}$, and is given by
\begin{equation}
\label{eq-phi} \phi({\bf x}) - \phi({\bf y}) = 0.
\end{equation}
Further, the image of $S \times_S S$ in $S^{[2]}$ is obtained
through the quotient $S \times S \longrightarrow (S \times S)
\slash \mu_2$ by the cyclic group $\mu_2$, which interchanges the
factors. Note that polynomials involved in \eqref{eq-phi} are not
$\mu_2$-invariant and so the image of $S \times_S S$ does not have
a good scheme structure inside of $S^{[2]}$ (like l.c.i. property
for instance). Thus to define the incidence locus $J$ from
\cite{D} properly one needs to do this in $S \times S \times
|\mathcal{O}_S(k)|$ first and then descend it,
$\mu_2$-equivariantly, to $S^{[2]}$.

Now observe that $S \times_S S \subset S \times S$ may be
\emph{reducible} (as in the preceding example). In this case,
projecting $J \subset S \times S \times |\mathcal{O}_S(k)|$ to one
of the irreducible components of $S \times_S S$ corresponds to
fixing just \emph{one} point in the fiber of $\phi$, which implies
that generic fiber of this projection has dimension $\dim
|\mathcal{O}_S(k)| - 1$ (again this is best seen for the above
example). Hence $\dim J = \dim |\mathcal{O}_S(k)| + 1$ ---
contrary to what was claimed in \cite{D}.

In short, \emph{set-theoretical} constructions applied to
\emph{schemes} need not always give what one expects, thus
creating a point of confusion.

\bigskip

(iv) \emph{\bf On p. 308}: In the final paragraph, the condition
$k \ge 1$ should be replaced by $k \ne 0$, with rational $k$; the
corresponding divisorial identities are understood numerically
then. This does not affect the rest of the argument.

\bigskip

\bigskip

{\bf Funding.} The work was carried out at the Center for Pure
Mathematics (MIPT) and was partially supported by the Russian
Science Foundation under grant 25-21-00083
(https://rscf.ru/project/25-21-00083/).

\end{document}